\documentclass[a4paper,11pt,leqno]{amsart}

\usepackage{amsmath}
\usepackage{amsthm}
\usepackage{amsfonts}
\usepackage{amssymb}

\usepackage{tikz}
\usepackage{enumerate}
\usepackage{colonequals}

\newcommand{\inv}{^{-1}}
\newcommand{\id}{\operatorname{id}}
\newcommand{\eps}{\varepsilon}
\newcommand{\N}{\mathbb N}
\newcommand{\Z}{\mathbb Z}
\newcommand{\R}{\mathbb R}
\newcommand{\C}{\mathbb C}
\newcommand{\bS}{\mathbb S}
\newcommand{\diam}{\operatorname{diam}}

\newtheorem{theorem}{Theorem}[section]
\newtheorem*{theorem*}{Theorem}{\bf}{\it}
\newtheorem{proposition}[theorem]{Proposition}
\newtheorem*{proposition*}{Proposition}{\bf}{\it}
\newtheorem{lemma}[theorem]{Lemma}
\newtheorem*{lemma*}{Lemma}{\bf}{\it}
\newtheorem{corollary}[theorem]{Corollary}
\theoremstyle{definition}
\newtheorem{definition}[theorem]{Definition}
\newtheorem*{definition*}{Definition}
\theoremstyle{remark}
\newtheorem{remark}[theorem]{Remark}
\numberwithin{equation}{section}
\newtheorem{example}[theorem]{Example}
\numberwithin{equation}{section}

\title{A Newman property for BLD-mappings}
\author{Rami Luisto}
\address{Department of Mathematics and Statistics, P.O. Box 35, FI-40014 University of Jyv\"askyl\"a, Finland
\and
Department of Mathematical Analysis, Sokolovska 83, Praha 8, 186 75, Charles University in Prague}
\email{rami.luisto@gmail.com}
\thanks{
  The first author was partially supported by a grant of the Finnish Academy of Science and Letters,
  the Academy of Finland
  (grant 288501 `\emph{Geometry of subRiemannian groups}')
  and by the European Research Council
  (ERC Starting Grant 713998 GeoMeG `\emph{Geometry of Metric Groups}').
}

\date{\today}

\begin{document}

\begin{abstract}
  We define a Newman property for BLD-mappings and 
  prove that for a BLD-mapping between generalized manifolds equipped with complete path-metrics
  this property is equivalent to the branch set being porous when the codomain is LLC.
\end{abstract}

\maketitle

\section{Introduction}
\label{sec:Intro}

The class of BLD-mappings was introduced by Martio and V\"ais\"al\"a in
\cite{MartioVaisala} as subclass of quasiregular mappings that 
preserve solutions of certain elliptic partial differential 
equations. In that paper Martio and V\"ais\"al\"a 
showed, among other results, that the class
of BLD-mappings has several equivalent definitions. 
In this paper we use the following geometric definition applicable
in any path-metric space.
For the definitions of a branched cover, the length of a path
and path-metric spaces, see Section \ref{sec:Preli}.
\emph{
  Given $L \geq 1$, a branched cover $ f \colon M \to N$ between metric spaces
  is a mapping of \emph{Bounded Length Distortion},
  or ($L$-)BLD for short, if
  for all paths $\gamma \colon [0,1] \to X$, we have}
  \begin{align*}
    L \inv \ell(\gamma)
    \leq \ell(f \circ \gamma)
    \leq L \ell(\gamma).
    \tag{BLD}
  \end{align*}
We denote by $B_f$ the \emph{branch set} of a branched cover $f$, i.e.\ the set of
points in the domain where the mapping fails to be a local homeomorphism.
A basic example of a BLD-mapping is the planar winding map $w \colon \C \to \C$ ,
$z \mapsto \frac{z^2}{| z |}$ along with its higher dimensional analogues
$w \times \id_{\R^n} \colon \R^{n+2} \to \R^{n+2}$. These examples 
have branch sets of both zero measure and small topological dimension; 
such a property holds in general for BLD-mappings.
Indeed, for an $L$-BLD mapping $f$ between $n$-dimensional Euclidean domains the topological
dimension of the branch set is at most $n-2$ by 
the Cernavskii-V\"ais\"al\"a theorem, see e.g.\ \cite{Vaisala},
and $\dim_{\mathcal{H}}(B_f) < n - \eps(n,L)$, see 
e.g.\ \cite{BonkHeinonenBranch} and \cite[Section 4.27]{MartioVaisala}.
The topological dimension of the branch set is not completely understood,
even in dimension three; for example the Church-Hemmingsen conjecture in \cite{ChurchHemmingsen1} asks
when the Cernavskii-V\"ais\"al\"a theorem is exact.
On the other hand, the estimate for the Hausdorff dimension is strict in the 
sense that for any $n \geq 3$ and $\eps > 0$ there exists a BLD-mapping $g \colon \R^n \to \R^n$ with
$\dim_{\mathcal{H}}(B_g) \geq n - \eps$, see e.g.\ \cite{HeinonenRickman} and \cite[Section 4.27]{MartioVaisala}.

With the Hausdorff dimension of the branch set bounded away from $n$, it 
follows that the branch set of a BLD-mapping $f \colon \Omega \to \R^n$ has $n$-measure zero.
It is an interesting question, posed for example by Heinonen and Rickman in \cite[Remarks 6.7 (b)]{HeinonenRickman},
that in what generality does this hold. The measure of the branch set is also connected to
certain limit properties for BLD-mappings,
see \cite[Remark 4.5]{Luisto-Characterization} for commentary on the topic.
There are positive results concerning this question for BLD-mappings outside the 
setting of Riemannian manifolds;
Heinonen and Rickman show in \cite[Theorem 6.4]{HeinonenRickman} that for BLD-mappings
between quasiconvex spaces of type A (see Section \ref{sec:Preli} for definitions)
the branch set of a BLD-mapping has measure zero. It is not completely known
which assumptions of \cite[Theorem 6.4]{HeinonenRickman} can be relaxed 
-- especially if, in addition, the spaces are
assumed to be generalized manifolds. We construct in Section \ref{sec:example}
a BLD-mapping between Ahlfors 2-regular metric spaces with $\mathcal{H}^2(B_f) > 0$
which exemplifies some of the restrictions for possible positive results in the future and the limitations
of the methods in the current paper.

Our approach for the study of the measure theoretic size of the branch
set is motivated by the fact that for a BLD-mapping
$f \colon \Omega \to \R^n$ the branch set is \emph{locally porous},
i.e.\ for any point $x_0 \in B_f$ there exists a constant $\delta > 0$ 
and a neighborhood $U$ of $x_0$ such that for any ball
$B(x,r) \subset U$ there is a point $y \in B(x,r)$ for which
$B_f \cap B(y,\delta r ) = \emptyset$. In an Ahlfors $n$-regular
metric measure space a porous set has measure zero, so 
in such a setting the question of Heinonen and Rickman can be 
approached via the porosity of the branch set.
This approach has had previous success; in
\cite{OnninenRajala} Onninen and Rajala study the branch set
of quasiregular mappings between generalized manifolds of
type A via a stratification to porous sets
and reach estimates on the Hausdorff dimension of the branch set.
Our approach to this question in the setting of BLD-mappings
is motivated by the proof of \cite[Theorem 4.25]{MartioVaisala};
the theorem in question implies that the branch set of a Euclidean
BLD-mapping is porous. It turns out that the porosity is connected
to the following concept of a Newman property for a BLD-mapping:
\emph{$f$ satisfies an $\eps_0$-Newman type property at $x_0 \in X$
if there is a neighborhood $U$ of $x_0$ and a constant $\eps_0 > 0$ 
such that in any ball $B(z,r) \subset U$ with $z \in B_f$,
there exists a point $w \in B(z,r)$ such that}
\begin{align*}
  \diam\left(B(z,r) \cap f \inv \{ f(w) \} \right)
  \geq \eps_0 r.
\end{align*}
The term Newman property is motivated by an injectivity 
radius result of Newman from 1931, \cite{Newman},
which essentially states that if a finite group $G$
acts nontrivially on a connected topological manifold $M$
then at least one of the orbits has a diameter of at least
$\eps = \eps(M) > 0$. Newman's result was generalized by
McAuley and Robinson in \cite{McAuleyRobinson-83} to state that for any manifold $M$
there is a constant $\eps > 0$ such that any surjective
branched cover $f \colon M \to N$ onto a manifold $N$
is either a homeomorphism or has a fiber with a diameter of
at least $\eps$. Later on in \cite{McAuleyRobinson-84}
they elaborated on the topic and coined the term 
Newman property; note, however that our terminology 
differs slightly from theirs.

For BLD-mappings the Newman property is strongly connected to the
porosity of the branch set and
our main theorem is as follows. For the definition of the terms
see Section \ref{sec:Preli}.
\begin{theorem}\label{thm:Main}
  Let $f \colon M \to N$ be an $L$-BLD mapping between generalized $n$-manifolds
  both equipped with a complete path-metric.
  Suppose that the mapping $f$ satisfies a Newman property at $x_0 \in M$.
  Then the branch set of $f$ is locally porous at $x_0$.
  
  Furthermore if $N$ is LLC, then these conditions are equivalent.
\end{theorem}

We note that natural examples of generalized manifolds satisfying the
local LLC condition arise from the classes of Riemannian and subRiemannian manifolds,
(see Section \ref{sec:Preli}).
The theory of BLD-mappings from a generalized manifold onto Euclidean spaces
and manifolds have been studied e.g.\ by Heinonen and Rickman in \cite{HeinonenRickman},
and we record the following corollary in this spirit.
\begin{corollary}
  Let $f \colon M \to N$ be an $L$-BLD mapping where
  $M$ is a generalized $n$-manifold and $N$ a (sub)Riemannian $n$-manifold,
  both equipped with a complete path-metric.
  Then the following are equivalent for $x_0 \in M$.
  \begin{enumerate}
  \item The mapping $f$ satisfies a Newman property at $x_0$.
  \item The the branch set of $f$ is locally porous at $x_0$.
  \end{enumerate}
\end{corollary}

Finally we wish to note that Newman-type properties have seen previous use
in the study of geometric mappings in the form of the McAuley-Robinson
theorem in \cite{McAuleyRobinson-84}, see
e.g.\ in \cite{Rajala-LowerBound} and \cite{KoskelaOnninenRajala}.
Furthermore, recent work of Guo and Williams partly relies on a generalization
of the McAuley-Robinson theorem --
this generalization, \cite[Corollary 4.8]{GuoWilliams1}, is essentially
a McAuley-Robinson theorem for branched covers between generalized
manifolds with the extra assumptions that the domain satisfies a so called $n$-LLC property
and the codomain has bounded turning. Using this they show that for a BLD-mapping $f \colon M \to N$
between generalized $n$-manifolds equipped with Ahlfors regular measures
the branch set has zero measure under the aforementioned additional assumptions
of $M$ being $n$-LLC and $N$ having bounded turning. Beyond
the use of a McAuley-Robinson type theorem their methods
are quite dissimilar compared to the current manuscript; see
\cite{GuoWilliams2} and \cite{GuoWilliams1} for details.

\section{Preliminaries}
\label{sec:Preli}

A mapping between topological spaces is said to be \emph{open} if the image
of every open set is open and \emph{discrete} if the point inverses are discrete sets in the domain.
A continuous, discrete and open mapping is called a \emph{branched cover}.
We follow the conventions of \cite{Rickman-book} and say that $U \subset X$ is a \emph{normal domain} (for
$f$) if $U$ is a precompact domain such that $\partial f (U) = f (\partial U)$. A normal domain $U$
is \emph{a normal neighborhood} of $x \in U$ if $\overline{U} \cap f \inv \{ f(x) \} = \{ x \}$.
By $U(x,f,r)$ we denote the component
of the open set $f \inv B_Y(f(x),r)$ containing $x$.
The following lemma guarantees the existence of normal domains,
for a proof see e.g.\ \cite[Lemma I.4.9, p.19]{Rickman-book} 
or \cite[Lemma 5.1]{Vaisala}.
\begin{lemma}\label{lemma:TopologicalNormalDomainLemma}
  Suppose $f \colon X \to Y$ is a branched cover between two locally compact
  and complete metric spaces. 
  Then for every point $x \in X$ 
  there exists a radius $r_x > 0$ such that
  $U(x,f,r)$
  is a normal neighborhood of $x$ for
  any $r \in (0,r_x)$.
\end{lemma}
Another topological result we need states that the branch set cannot
have interior points in a very general setting.
The following lemma can be found e.g.\ from \cite[Theorem 3.2]{Vaisala}.
\begin{lemma}\label{lemma:NoInterior}
  Suppose $f \colon X \to Y$ is a branched cover between two locally compact
  and complete metric spaces. 
  Then the branch set $B_f$ has no interior points.
\end{lemma}

In a metric space $X$ the
\emph{length} $\ell(\gamma)$ of a path 
$\gamma \colon [0,1] \to X$ is defined as
\begin{align*}
  \ell(\gamma)
  \colonequals 
  \left\{
    \sum_{j=1}^{k} d(\gamma(t_{j-1}),\gamma(t_j)) \mid 
    0 = t_0 \leq \ldots \leq t_k = 1 
  \right\}.
\end{align*}
Paths with finite and infinite length are called
rectifiable and unrectifiable, respectively, and
a metric space $(X,d)$ is 
a \emph{path-metric space} if
\begin{align*}
  d(x,y)
  = \inf 
  \left\{ 
    \ell(\gamma) \mid \gamma \colon [0,1] \to X, \gamma(0) = x , 
    \gamma(1) = y 
  \right\}
\end{align*}
for all $x,y \in X$.
Similarly, a metric space $(X,d)$ is ($C$-)quasiconvex
if for all $x,y \in X$ there exist a path $\gamma \colon [0,1] \to X$
with $\gamma(0)=x$, $\gamma(1)=y$ and $\ell(\gamma) \leq C d(x,y)$.
A $1$-quasiconvex space is called a \emph{geodesic space}.

The following proposition is a part of Theorem 1.1 in \cite{Luisto-Characterization}.
\begin{proposition}\label{prop:BLDDefinitions}
  Let $f \colon M \to N$ be a continuous mapping
  between two complete locally compact
  path-metric spaces and $L \geq 1$. Then the following are
  equivalent:
  \begin{enumerate}[(i)]
  \item $f$ is an $L$-BLD-mapping, and
  \item $f$ is a discrete $L$-LQ-mapping.
  \end{enumerate}
\end{proposition}
Here a mapping $f \colon M \to N$ is \emph{$L$-Lipschitz Quotient}, $L$-LQ for short,
if for all $x \in X$ and $r > 0$ we have
\begin{align*}
  B(f(x),L \inv r) 
  \subset f B(x,r)
  \subset B(f(x),L r).
\end{align*}
We will use the fact that $L$-BLD mappings are $L$-LQ maps repeatedly in the arguments.

A natural setting for the study of branched covers is that of
generalized manifolds; see \cite[Section I.1]{HeinonenRickman} for details.
\begin{definition}
  A locally compact Hausdorff space $M$ 
  which is both connected and locally connected
  is called a \emph{generalized $n$-manifold}, if
  \begin{enumerate}
  \item the cohomological dimension $\dim_\Z M$ is at most $n$, and
  \item the local cohomology groups of $M$ are equal to $\Z$ in degree $n$ 
    and zero in degree $n-1$.
  \end{enumerate}
\end{definition}
The importance of generalized manifolds arises from the fact that
for a branched cover $f \colon M \to N$ between generalized $n$-manifolds
topological index and degree theory are applicable.
Especially topological index and degree theory imply the
following lemma, see e.g.\ \cite[Proposition I.4.10. (1)]{Rickman-book} for a proof.
Here we denote by $N(y,f,U)$ the amount of pre-images of $y$ under $f$ in
$U$, and set $N(f,U)$ to be the maximum of $y \mapsto N(y,f,U)$ in $U$.
\begin{lemma}\label{lemma:Equidistribution}
  Let $f \colon M \to N$ be a branched cover between two generalized
  $n$-manifolds. Then for any normal domain $U \subset M$,
  every point $y \in fU \setminus fB_f$ satisfies
  $N(y,f,U) = N(f,U) < \infty$.
\end{lemma}

For BLD-mappings between generalized manifolds the amount of pre-images can be quantitatively
bounded in some cases. This in turn gives rise to further geometrical control for the mapping.
The following lemma is a variation of known family of results in this vein, compare e.g.\ to
\cite[Theorems 4.12 and 4.14]{MartioVaisala}, \cite[Lemma 3.2]{LuistoPankka} and \cite[Theorem 1.4]{Luisto-NoteOnLocalToGlobal}.
We give a short proof for the convenience of the reader. 
\begin{lemma}\label{lemma:DiameterBound}
  Let $f \colon M \to N$ be an $L$-BLD mapping between two generalized
  $n$-manifolds equipped with complete path-metrics and let $U \subset M$ be normal domain. Then for any $r_0 > 0$
  and any continuum $K \subset fU$ with $B(K,r_0) \subset fU$ and $(\diam K) < r_0$,
  all the components of $U \cap f \inv K$ have diameter of at most $2 L (N(f,U)+1)(\diam K)$.
\end{lemma}
\begin{proof}
  Fix a continuum $K \subset fU$ with $B(K,r_0) \subset fU$ and $\diam K < r_0$ for some $r_0 > 0$.
  Take a component $C$ of $U \cap f \inv K$ and note that if $\diam C > 2L (N(f,U)+1) (\diam K)$,
  then there exists $k \colonequals N(f,U) + 1$ points $x_1, \ldots, x_k \in K$ such that
  $d(x_i,x_j) > 2L (\diam K)$ for any $i \neq j$.
  
  Now we note that since $f$ is $L$-LQ, the images $fB(x_i, L (\diam K))$ of the disjoint balls $B(x_i, L (\diam K))$
  all contain a ball of radius $\diam K$. Since $f(x_i) \in K$ for each $i$, this implies that there exists a point
  $y_0 \in \cap_{i=1}^k fB(x_i,L (\diam K))$, and in particular $N(f,U,y_0) \geq k = N(f,U) + 1$. This is a contradiction,
  and so the original claim holds.
\end{proof}

Another crucial notion for us is local linear contractibility.
\begin{definition}
  A metric space $X$ is \emph{locally linearly contractible} (LLC) if 
  for any compact set $K \subset X$ there exists constants $C \geq 1$ and $ r_0 > 0$ such
  that for all $x \in K$ and $r \in (0,r_0)$ the ball $B(x,r) \subset X$ is contractible in $B(x,Cr)$.
\end{definition}
Natural examples of LLC spaces arise from Euclidean spaces, Riemannian manifolds (via the local bilipschitz equivalence
of Riemannian manifolds and Euclidean spaces given by the exponential map) and subRiemannian manifolds (via the ball-box theorem,
see e.g.\ \cite[Theorem 2.10, p.29]{Montgomery}).

In the literature the LLC condition is often used as a part of the definition of \emph{spaces of type A.}
We do not use spaces of type A neither in the statement nor in the proof of our main
theorem, but describing the properties of the example constructed
in Section \ref{sec:example} is more fluent with this standard terminology.
We refer the reader to \cite[Section I.1]{HeinonenRickman} for more details.
\begin{definition}
  For a metric space $X$ and an integer $n \geq 2$ consider the following properties
  $X$ may or may not have.
  \begin{enumerate}[({A}1)]
  \item $X$ is $n$-rectifiable and has locally finite Hausdorff $n$-measure.
  \item $X$ is $n$-Ahlfors regular.
  \item $X$ is locally bilipschitz embeddable in a Euclidean space.
  \item $X$ is LLC, i.e.\ locally linearly contractible.
  \end{enumerate}
  A space $X$ is said to be \emph{of type A} if it satisfies all of these
  properties and \emph{of type (A$123$)} if it satisfies properties
  (A1), (A2) and (A3).
\end{definition}

The proofs of the main theorems rely on the concept of an \emph{blow up} 
of a given metric space. The blow up of a metric space at a point
can be seen as a generalization of a tangent space; indeed in the setting of smooth manifolds,
a blow up at a point $x_0$ is isometric to the tangent plane
of the space at $x_0$. For the general theory of blow-ups we
refer to \cite{Kapovich-Book}, for their interaction with BLD- and LQ-mappings
see e.g.\ \cite[Section 3]{LeDonnePankka} and \cite[Section 4]{Luisto-Characterization}.

\section{Main results}
\label{sec:Main}

We begin by proving the first implication of the main theorem
in the form of the following Proposition \ref{Proposition:NewmanImpliesPorosity}.
\begin{proposition}\label{Proposition:NewmanImpliesPorosity}
  Let $f \colon M \to N$ be an $L$-BLD-mapping between two generalized 
  manifolds both equipped with a complete path-metric.
  Suppose that $f$ satisfies an $\eps_0$ Newman property at $x_0$.
  Then $B_f$ is locally porous at $x_0$.
\end{proposition}
\begin{proof}
  Suppose $B_f$ is not locally porous at $x_0$. 
  Then by the definition of porosity we can fix a sequence $r_n \searrow 0$ such that
  \begin{align}\label{eq:nonporosity}
    B(x,\frac{1}{n} r_n) \cap B_f
    \neq \emptyset 
  \end{align}
  for all $x \in B(x_0,\frac{1}{2}r_n)$ and $n \in \N$.
  For each $j \in \N$, denote by $X_j$ and $Y_j$ the spaces $M$ and $N$ equipped with
  the metrics scaled by the factor $r_j \inv$, respectively. The $L$-BLD-mapping $f \colon M \to N$
  induces canonical $L$-BLD-mappings $f_j \colon X_j \to Y_j$ for each $j \in \N$. Note that by \eqref{eq:nonporosity}
  for all $j \in \N$ and $x \in B_{X_j}(x_0,\frac{1}{2})$, we have
  $B_{X_j}(x,j \inv) \cap B_f \neq \emptyset$.

  We refer to \cite[Section 4]{Luisto-Characterization} for the fact that
  there exists, after possibly passing to a subsequence, locally compact and 
  complete path-metric spaces $(\hat X,\hat x_0)$ and $(\hat Y,\hat y_0)$ together
  with an $L$-Lipschitz mapping $\hat f \colon \hat X \to \hat Y$ 
  such that in the sense of convergence of mapping packages,
  \begin{align*}
    \lim_{j \to \infty} ((X_j,x_0), (Y_j,f(x_0)), f_j) 
    = ((\hat X,\hat x_0), (\hat Y,\hat y_0), \hat f) 
  \end{align*}
  and $\hat f (\hat x_0) = \hat y_0$.
  Note that for any normal neighborhood $U$ of $x_0$ any point 
  in $fU$ has at most finitely many pre-images in $U$
  since $f$ is discrete and $U$ is precompact.
  Thus by Lemma \ref{lemma:Equidistribution} there exists
  $N_0 \in \N$ such that
  $U \cap f \inv \{ f(x) \} \leq N_0$ for all $x \in U$.
  This implies that for any $r>0$ and all $j \in \N$ large enough,
  \begin{align}\label{eq:Bound}
    \# \left( B_{X_j}(x_0,r) \cap f \inv \{ f(x) \} \right)
    \leq N_0
  \end{align}
  for all $x \in X_j$. Thus the mappings $f_j \colon X_j \to Y_j$ 
  have a uniform bound on their multiplicity, and 
  by \cite[Theorem 4.1]{Luisto-Characterization} the limiting
  map $\hat f \colon \hat X \to \hat Y$ is $L$-BLD. 
  (In the proof of \cite[Theorem 4.1]{Luisto-Characterization} the assumption 
  about having spaces of type A with uniform constants is only used
  to guarantee a uniform multiplicity bound (N), which we have here
  from \eqref{eq:Bound}.)

  We show next that $B_{\hat f} \supset B_{\hat X}(\hat x_0, \frac{1}{4})$.
  To this end, let $\hat y \in B_{\hat X}(\hat x_0, \frac{1}{4})$.
  Let $(y_j)$ be a sequence in $\prod_j X_j$ corresponding to $\hat y$ and
  fix $\eps \in (0,1/4)$. For each $j$ we study the ball $B_j \colonequals B_{X_j}(y_j,\eps)$.
  By the definition of the spaces $X_j$ and equation \eqref{eq:nonporosity},
  there exists $b_j \in B_{f_j} \cap B_j$
  for all large enough $j$. 
  Now by the Newman property there exists a constant $\delta_0 >0$ such that,
  for all $j$ large enough, there exists a point $z_j \in B_j$ for which
  \begin{align*}
    \diam(B(b_j,\frac{1}{2}) \cap f_j \inv \{ f(z_j) \}) 
    \geq \delta_0.
  \end{align*}
  From these pre-images of relatively large diameter 
  we extract two sequences, $(w_j)$ and $(w_j')$ with $d_{X_j}(w_j,w_j') \geq \delta_0$ for all
  $j \in \N$. After passing to a subsequence if necessary these sequences converge to 
  points $\hat w, \hat w ' \in B(\hat y, \eps) \subset \hat X$, respectively,
  such that $\hat d ( \hat w , \hat w') \geq \delta_0 > 0$.
  Furthermore the points $\hat w$ and $\hat w'$ map to the same point under $\hat f$. Thus $\hat f$ is not
  injective in $B(\hat y , \eps)$. Since the argument goes through for any $\eps \in (0, 1/4)$,
  we conclude that $\hat y \in B_{\hat f}$.

  Thus the limit map $\hat f$, which is an $L$-BLD mapping between
  locally compact and complete metric spaces contains a ball in its branch set. 
  This is a contradiction with Lemma \ref{lemma:NoInterior}, and the original claim holds.
\end{proof}

We show next that the converse of Proposition \ref{Proposition:NewmanImpliesPorosity}
holds under the extra assumption that the space $N$ is LLC.

The proof of Proposition \ref{prop:PorousToNewman}
relies on the following Zorich-type result for BLD-mappings
which we have not seen explicitly stated in the
literature, even though it is known to the experts in the field;
for a Euclidean version see e.g.\ \cite[Lemma 4.3]{MartioVaisala}.
\begin{lemma}\label{lemma:BLDZorich}
  Let $f \colon M \to N$ be an $L$-BLD-mapping 
  between generalized $n$-manifolds both equipped with a complete path-metric.
  Furthermore suppose that $N$ is LLC.  
  Then for each precompact domain $U \subset M$
  there exists constants $r_0 > 0$ and  $\phi > 0$ such that if $B(x,r) \subset U$ and
  $B(x,r) \cap B_f = \emptyset$ for some $x \in X$, $r \in (0,r_0)$, then $f|_{B(x, \phi r)}$ is injective.
\end{lemma}
\begin{proof}
  Let $U \subset M$ be a precompact domain. Since $\overline{U}$ is compact it can be covered
  with finitely many normal neighborhoods $U(x_j,f,r_j)$.
  Thus by Lebesgue's number lemma there exists a radius $\delta > 0$ such that any 
  ball in $N$ of radius $\delta$ is contained in at least one of the balls $B(f(x_j),r_j)$.
  We may thus assume that $U$ itself is a normal neighborhood by requiring $r_0$ to be small enough.
  Let $D$ be the LLC-constant of the space $N$ with respect to the
  compact set $f\overline{U}$. We denote $k \colonequals N(f,U)$ and set
  $\phi \colonequals (16L^2(k+1)D)\inv$ -- note that $\phi \leq 2\inv$. Finally let $r_0$ be so small
  that $2 \phi L r_0$ is below the LLC-radius of the space $N$ with respect to the compact
  set $f\overline{U}$.
  
  Fix a ball $B(x_0,r) \subset U$ with $B(x_0,r) \cap B_f = \emptyset$
  for some $x_0 \in U$.
  Towards contradiction assume that there exists disjoint points
  $a,b \in B(x_0, \phi r)$ with $f(a) = f(b)$. By the Hopf-Rinow theorem (see e.g.\ \cite[p.\ 9]{Gromov})
  the locally compact and complete path-metric space $M$ is a geodesic space, and so we may fix a geodesic
  $\beta \colon [0,1] \to M$ with $\beta(0) = a, \beta(1) = b$.
  Note that $\diam |\beta| \leq 2\phi r$, so $\beta \colon [0,1] \to U$.
  Since $f$ is $L$-Lipschitz,
  \begin{align*}
    \diam (f| \beta |)
    \leq L \diam (|\beta|)
    \leq 2 L \phi r
    = \frac{r}{8L(k+1)D}.
  \end{align*}
  Since $f(a) = f(b)$, $f \circ \beta \colon [0,1] \to B(f(a), 2L \phi r)$ is a loop based on $f(a)$.
  Thus, be the local LLC-property of $N$, there is a homotopy
  $H \colon [0,1]^2 \to B(f(a),2 D L \phi r) $
  taking the loop $f \circ \beta$ to a constant path $t \mapsto f(a)$.

  Now by the path lifting theorem \cite[Theorem II.3.2]{Rickman-book} each of the
  paths $t \mapsto H(t,s)$, $s \in [0,1]$, has a unique total 
  lift $\alpha_s$ in the normal domain $U$ starting from $a$. On the other hand
  we note that for each $ s \in [0,1]$,
  \begin{align*}
    \diam (f|\alpha_s|)
    \leq \diam (H([0,1]^2))
    \leq 4 D L \phi r
    = \frac{r}{4L(k+1)}
  \end{align*}
  and so by Lemma \ref{lemma:DiameterBound} each of the lifts $\alpha_s$
  has diameter of at most
  \begin{align*}
    2L(k+1)\frac{r}{4L(k+1)}
    = r/2.
  \end{align*}
  In particular this implies that each of the lifts $\alpha_s$
  is contained in the ball $B(x_0,r)$, since
  $a \in B(x_0,\phi r) \subset B(x_0,r/2)$. Thus the lifts do not intersect
  any branch points. From this we conclude
  that the homotopy $H$ lifts to a homotopy $\tilde H$
  in $B(x_0,r)$ contracting the path $\beta$ to a constant path
  while keeping the endpoints $a$ and $b$ fixed. This is
  a contradiction so the original claim holds true:
  $f$ is injective in $B(x,\phi r)$.
\end{proof}

\begin{proposition}\label{prop:PorousToNewman}
  Let $f \colon M \to N$ be an $L$-BLD-mapping between
  a generalized $n$-manifolds equipped with a complete path-metrics and
  let $N$ be locally LLC.
  Suppose $B_f$ is locally porous at $x_0 \in B_f$.
  Then $f$ satisfies a Newman property at $x_0$.
\end{proposition}
\begin{proof}
  Since $B_f$ is locally porous at $x_0 \in B_f$, there exists by the definition of porosity
  a neighborhood $U$ of $x_0$ and a constant $\delta > 0$ such that
  for all $x \in U$ with $B(x,r) \subset U$, there exists a point
  $y \in B(x,r)$ with
  \begin{align*}
    B(y,\delta r) 
    \subset B(x,r) \setminus B_f.
  \end{align*}
    Let $B(z,s) \subset U$ with $z \in B_f$. 
  By the definition of the branch, $f$ cannot be locally injective at $z$. Thus
  by Lemmas \ref{lemma:TopologicalNormalDomainLemma} and \ref{lemma:Equidistribution}
  we may assume both that $U$ is a normal domain and that all points in 
  $fU \setminus fB_f$ have an equal amount of pre-images in $U$.
  Note that by the Cernavskii-V\"ais\"al\"a theorem, see \cite{Vaisala}, $fB_f$ has no interior points
  and so its complement in $fU$ is not empty.
  
  By the local porosity of $B_f$, there
  exists a point $w \in B(z,(2L^2)\inv s)$ such that
  \begin{align*}
    B(z,\delta (2L)\inv s) \subset B(z,(2L^2)\inv s) \setminus B_f.
  \end{align*}
  By Lemma \ref{lemma:BLDZorich} there exists a constant
  $\phi > 0$ for which $f|_{B(w,\phi \delta (2L^2)\inv s)}$ is injective.
  The map $f$ is an open mapping, so the set
  $f B(w, \phi \delta (4L^2)\inv s)$ is open and must contain a point $y \in fU \setminus fB_f$.
  Thus the point $y$ has at least two pre-images in $U$, one of them in 
  $B(w,\phi\delta (4L^2)\inv s)$. We claim that there is another pre-image of
  $y$ in the ball $B(z,s)$. Indeed, fix a geodesic $\beta$ connecting $y$ to $f(z)$.
  Since $f$ is $L$-BLD, the path $\beta$ has length of at most $(2L)\inv s$.
  By the path-lifting theorem \cite[Theorem II.3.2]{Rickman-book}, the path $\beta$ has
  at least two total lifts in $U$ starting from $z$ and these two lifts
  have disjoint endpoints in $f \inv \{ y \}$ since $y \notin fB_f$. Furthermore since $f$
  is $L$-BLD, these two lifts have length of at most $2 \inv s$.
  Thus we have
  \begin{align*}
    \diam (B(z,s) \cap f \inv \{ y \} )
    \geq \phi \delta (4L^2)\inv s
  \end{align*}
  since $f|_{B(w,\phi \delta (2L^2)\inv s)}$ is injective,
  and so $f$ has the Newman property at $x_0$.
\end{proof}

\begin{proof}[Proof of Theorem \ref{thm:Main}]
  The proof of the theorem now follows by combining Propositions
  \ref{Proposition:NewmanImpliesPorosity} and
  \ref{prop:PorousToNewman}.
\end{proof}

\subsection{A BLD-mapping with a branch set of positive measure}
\label{sec:example}

\begin{figure}[h!]
  \centering
  \resizebox{\textwidth}{!}{
    \begin{tikzpicture}
      
      \begin{scope}[scale=1]
        
          \draw[fill, black!5]  (2,-2)--(-6,-2)--(-2,2)--(6,2)--cycle;
          \draw[fill, black!0] (0,-1)--(-2,-1)--(0,1)--(2,1)--cycle;
          \draw (0,-1)--(-2,-1)--(0,1)--(2,1)--cycle;

          \begin{scope}[yscale=1]
            \draw[fill,black!10] (0,0) circle (1);          
            \draw (0,0) circle (1);
            \begin{scope}[yscale=0.25]
              \draw[dashed,domain=0:180] plot ({1*cos(\x)},{1*sin(\x)});
              \draw[domain=180:360] plot ({1*cos(\x)},{1*sin(\x)});
            \end{scope}
          \end{scope}

          \draw[fill,black!5] (1,0)--(2,0)--(1,-1)--(0,-1)--cycle;
          \draw (0,-1)--(1,0);
          \draw[dashed] (0,1)--(-1,0);
          \draw[dotted]  (2,-2)--(-6,-2)--(-2,2)--(6,2)--cycle;
          \draw[fill] (-1,0) circle (0.04);
          \draw[fill] (1,0) circle (0.04);

          \begin{scope}[yscale=1]
            \draw[dashed] (0,0) circle (1);
          \end{scope}

          \draw[fill,white] (-2,1) circle (0.5);
          \node at (-2,1) {$\id$};

      \end{scope}
    \end{tikzpicture}
  }
  \caption{The auxiliary domain $W$ in the construction of the mapping in Example \ref{ex:NoPorosity}.}
  \label{fig:NoPorosity}
\end{figure}
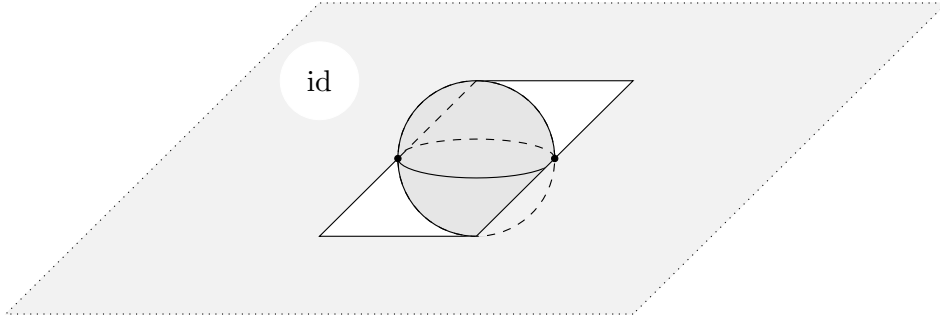

In this section we construct a compact and complete path-metric space $X$
equipped with an Ahlfors $2$-regular measure together with
a $2$-BLD mapping $f \colon X \to X$ such that $\mathcal{H}^2(B_f) > 0$.
Furthermore the constructed space $X$ is of type (A123)
and contains an open dense set which is a $2$-manifold,
though it is not a generalized manifold as it has
local cut-points.

The construction of the space $X$ and the mapping $f \colon X \to X$ in this section is given as a limit
of so called pointed mapping packages -- even though for our purposes this terminology is
a slightly excessive, it enables us to refer directly
to previous work in \cite{Luisto-Characterization}.
Here a \emph{pointed mapping package} is a triple
$\left((X,x_0),(Y,y_0),f\right)$ where $X$ and $Y$
are locally compact and complete 
path-metric spaces
having fixed base-points $x_0 \in X$, $y_0 \in Y$,
and $f \colon X \to Y$ is a continuous
mapping satisfying $f(x_0) = y_0$.
For the definition of the convergence of a sequence of pointed
mapping packages we refer to
\cite[Definition 8.18 and Lemma 8.19]{DavidSemmes},
see also
\cite[Definition 3.8]{KleinerMacKay} and
\cite[Definition 2.1]{GuyCDavid-Tangents}.
In essence, the definition states that
$(X_j, x_j) \to (X,x_0)$ and
$(Y_j, y_j) \to (Y,y_0)$ in the Gromov-Hausdorff sense
and $f_j \to f$ in a natural pointwise manner; indeed for
constant sequences $(X_j,x_j)$ and $(Y_j,y_j)$ the convergence is just the
pointwise convergence of mappings, see again e.g.\ \cite{DavidSemmes}.

Before constructing the sequence of pointed mapping packages, we define some auxiliary 
concepts. In the two-sphere $\bS^2$ we define the \emph{winding map} as the restriction
of the winding map $w \colon \R^n \to \R^n$, $w(r,\theta,z) = (r,2\theta,z)$.
Next we set
\begin{align*}
  W 
  = \bS^2 \cup \left( \{ 0 \} \times \left( [-4,4]^2  \setminus (-2,2) \times (-1,1) \right) \right)
  \subset \R^3,
\end{align*}
see Figure \ref{fig:NoPorosity},
and let $g \colon W \to W$ be a mapping which equals identity
the rectangular annulus, and the $2$-to-1 winding map on the sphere $\bS^2$.
In particular we note that $g$ has the same ``boundary data'' as
$\id \colon [-4,4]^2 \to [-4,4]^2$, even though these mappings are defined on different domains.
Note also that the mapping $g$ has exactly two branch points at the poles
of $\bS^2$. Furthermore a straightforward calculation shows that $g$ is $2$-BLD.

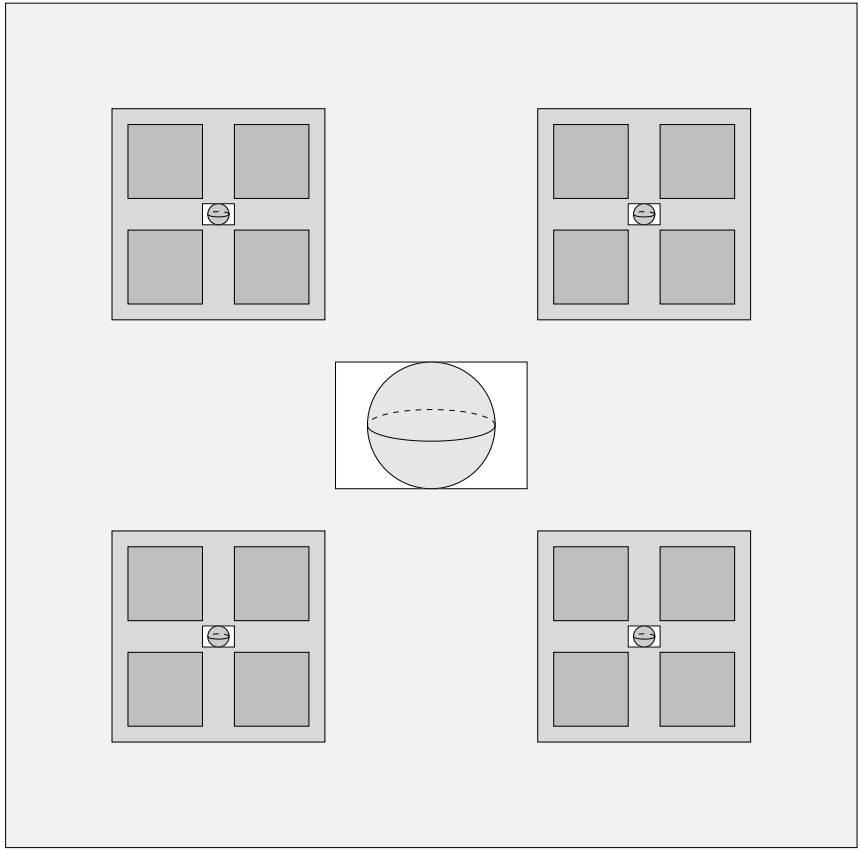
\begin{figure}[h!]
  \centering
  \resizebox{0.9\textwidth}{!}{
    \begin{tikzpicture}

      \begin{scope}[scale=0.8]

        \draw[fill,black!5] (-10,-10)--(10,-10)--(10,10)--(-10,10)--cycle;
        \draw (-10,-10)--(10,-10)--(10,10)--(-10,10)--cycle;

        \begin{scope}[scale=1.5]
          \draw[fill,white] (-1.5,-1)--(1.5,-1)--(1.5,1)--(-1.5,1)--cycle;
          \draw (-1.5,-1)--(1.5,-1)--(1.5,1)--(-1.5,1)--cycle;
          \begin{scope}[yscale=1]
            \draw[fill,black!10] (0,0) circle (1);
            \draw (0,0) circle (1);
            \begin{scope}[yscale=0.25]
              \draw[dashed,domain=0:180] plot ({1*cos(\x)},{1*sin(\x)});
              \draw[domain=180:360] plot ({1*cos(\x)},{1*sin(\x)});
            \end{scope}
          \end{scope}

        \end{scope}

        \foreach \x in {5,-5}
        \foreach \y in {5,-5}
        {
          
          \begin{scope}[shift={(\x,\y)},scale=0.25]
            
            \draw[fill,black!15] (-10,-10)--(10,-10)--(10,10)--(-10,10)--cycle;
            \draw (-10,-10)--(10,-10)--(10,10)--(-10,10)--cycle;
            \draw[fill,white] (-1.5,-1)--(1.5,-1)--(1.5,1)--(-1.5,1)--cycle;
            \draw (-1.5,-1)--(1.5,-1)--(1.5,1)--(-1.5,1)--cycle;
            \begin{scope}[yscale=1]
              \draw[fill,black!20] (0,0) circle (1);
              \draw (0,0) circle (1);
              \begin{scope}[yscale=0.25]
                \draw[dashed,domain=0:180] plot ({1*cos(\x)},{1*sin(\x)});
                \draw[domain=180:360] plot ({1*cos(\x)},{1*sin(\x)});
              \end{scope}
              \foreach \a in {5,-5}
              \foreach \b in {5,-5}
              {
                \begin{scope}[shift={(\a,\b)},scale=0.35]
                  \draw[fill,black!25] (-10,-10)--(10,-10)--(10,10)--(-10,10)--cycle;
                  \draw (-10,-10)--(10,-10)--(10,10)--(-10,10)--cycle;
                \end{scope}            
              }
              
          \end{scope}
          
        \end{scope}          
      }
  
      \end{scope}
    \end{tikzpicture}
  }
  \caption{Cantor set modifications in Example \ref{ex:NoPorosity},
  pictured is part of the domain of $f_2 \colon X_2 \to X_2$.}
  \label{fig:Cantor}
\end{figure}

\begin{example}\label{ex:NoPorosity}
  We define the sequence $(f_j, X_j, x_j, Y_j, y_j)$
  recursively. Since for us $X_j = Y_j$ and $x_j = y_j = (0,0,0)$ for
  all $j\in \N$, we denote the mapping package
  as $(f_j, X_j)$ for brevity.
  
  First we set $X_0 = \{0\} \times \R^2$ and $f_0 = \id \colon X_0 \to X_0$.
  On $X_0$ we fix a sequence $(C_j)$ of nested collections of squares
  giving rise to a Sierpinski carpet of positive measure. We fist set
  $C_0 =\left\{ \{0\} \times [0,1]^2\right\}$
  and for each collection $C_n$ defined we take for every
  $K \in C_n$ four disjoint subsqueares of $K$, each with area
  $(1-4^{-n-1}) \mathcal{H}^2(K)$ and positioned symmetrically in $K$ (see Figure \ref{fig:Cantor}). This collection of $4^{n+1}$ squares is
  then denoted $C_{n+1}$. With this construction the set
  \begin{align*}
    C
    \colonequals \bigcap_{n \in \N} \bigcup_{K \in C_n} K
  \end{align*}
  is a Cantor set of positive measure.

  Suppose next that $(f_n,X_n)$ has been defined. We set
  $(f_{n+1},X_{n+1})$ equal $(f_n, X_n)$ outside the set
  $\cup C_{n+1}$. On each of the squares $S \in C_n$
  we fix a subsquare $V$ around the center of $S$ such that
  $V \cap \cup C_{n+1} = \emptyset$. Finally on each such
  square $V$ we replace the square $V$ and the identity mapping
  $f_{n}|_V$ with a scaled version of the mapping $g \colon W \to W$;
  see again Figure \ref{fig:Cantor}.
  
  The sequence $(f_j,X_j)$
  such defined is a Cauchy sequence with respect to the Hausdorff-Gromov convergence of
  the sequence $(X_j)$ and the $2$-BLD-mappings $(f_j)$ are $2$-Lipschitz.
  Furthermore, the mappings $f_j$ are
  of uniformly bounded multiplicity, i.e.\ for any $f_j$, $j \in \N$, any given point has at
  most two pre-images. Thus by the proof of Theorem 1.4 in \cite{Luisto-Characterization}
  there exists a subsequence of this sequence converging to 
  a $2$-BLD-mapping $f \colon X \to X$, where $X$ is the Gromov-Hausdorff
  limit of $(X_j)$.
  Finally a moment's thought shows that the branch set of $f$ contains the Cantor set $C$ with $\mathcal{H}^2(C) > 0$.
\end{example}

\begin{remark}
  With minor modifications, the mapping in Example
  \ref{ex:NoPorosity} can be modified to an $1$-BLD-mapping $Y \to X$,
  where $Y$ is bilipschitz equivalent to $X$; indeed, this follows by pulling
  back the path length structure
  of the codomain under $f$, see e.g.\ \cite[Theorem 1.8]{Aaltonen}.
\end{remark}

\bigskip

\subsection*{Acknowledgments}
We thank Pekka Pankka for discussions on the topic.
We also extend our thanks to the anonymous referee whose
comments have improved the exposition of this manuscript
-- in particular their suggestion for the proper
generality of the main theorem is gratefully acknowledged.

 
\def\cprime{$'$}\def\cprime{$'$}

\end{document}